\documentclass[11pt,reqno]{amsart}

\title[]{Global solutions for the Muskat problem in the scaling invariant Besov space $\dot B^1_{\infty, 1}$}

\author{Huy Q. Nguyen}
\address{Department of Mathematics, Brown University, Providence, RI 02912}
\email{hnguyen@math.brown.edu}

\usepackage[margin=1in]{geometry}
\usepackage{amsmath, amsthm, amssymb, mathrsfs, stmaryrd}
\usepackage{times}
\usepackage{color}
\usepackage{hyperref}

\usepackage[utf8]{inputenc}
\usepackage[T1]{fontenc}

\newcommand{\bq}{\begin{equation}}
\newcommand{\eq}{\end{equation}}
\newcommand{\bqa}{\begin{eqnarray*}}
\newcommand{\eqa}{\end{eqnarray*}}

\theoremstyle{plain}
\newtheorem{theo}{Theorem}[section]
\newtheorem{prop}[theo]{Proposition}
\newtheorem{lemm}[theo]{Lemma}
\newtheorem{coro}[theo]{Corollary}

\newtheorem{defi}[theo]{Definition}
\theoremstyle{definition}
\newtheorem{rema}[theo]{Remark}

\DeclareMathOperator{\cnx}{div}
\DeclareMathOperator{\RE}{Re}

\DeclareMathOperator{\supp}{supp}

\DeclareSymbolFont{pletters}{OT1}{cmr}{m}{sl}
\DeclareMathSymbol{s}{\mathalpha}{pletters}{`s}

\def\eps{\varepsilon}
\def\na{\nabla}

\def\lb{\llbracket}
\def\rb{\rrbracket}
\def\les{\lesssim}

\def\mez{\frac{1}{2}}
\def\tdm{\frac{3}{2}}

\def\Rr{\mathbb{R}}

\def\Zz{\mathbb{Z}}
\def\Cc{\mathbb{C}}

\def\cF{\mathcal{F}}
\def\cK{\mathcal{K}}

\def\cP{\mathcal{P}}
\def\cT{\mathcal{T}}
\def\cV{\mathcal{V}}
\def\ld{\lambda}

\def\p{\partial}
\def\na{\nabla}

\def\ka{\kappa}

\def\wt{\widetilde}

\def\B{\dot B_{\infty, 1}}

\numberwithin{equation}{section}

\begin{document}
\begin{abstract}
The one-phase and two-phase Muskat problems with arbitrary viscosity contrast are studied in all dimensions. They are quasilinear parabolic equations for the graph free boundary. We prove that small data in the scaling invariant homogeneous Besov space $\dot B^1_{\infty, 1}$ lead to unique global solutions. The proof exploits a new structure of the Dirichlet-Neumann operator which allows us to implement a robust fixed-point argument. As a consequence of this method, the initial data is only assumed to be in $\dot B^1_{\infty, 1}$ and the solution map is Lipschitz continuous in the same topology. For the general Muskat problem, the only known scaling invariant result was obtained in the Wiener algebra (plus an $L^2$ assumption) which is strictly contained in $\B^1$. 
\end{abstract}


\noindent

\maketitle
\section{Introduction}

\subsection{The Muskat problem} 
The two-phase Muskat problem concerns the dynamics of the free interface between two immiscible fluids permeating a porous medium. The fluids have different constant densities $\rho^\pm$ and, in general, different viscosities $\mu^\pm$.  Assume that the free interface is the graph of a time-dependent function $\eta(x, t)$, i.e.
\[
\Sigma_t=\{(x, \eta(x, t)): x\in \Rr^d\},
\]
 so that the fluids occupy the regions 
\[\label{Omega+}
 \Omega^+_t=\{(x, y)\in \Rr^{d+1}: y>\eta(x, t)\},\quad \Omega^-_t=\{(x, y)\in \Rr^{d+1}: y<\eta(t, x)\}.
\]
The fluid motion is modeled by Darcy's law:
\bq\label{Darcy:pm}
\mu^\pm u^\pm+\nabla_{x, y}p^\pm=-(0, \rho^\pm),\quad \cnx_{x, y} u^\pm =0\quad \text{in}~\Omega^\pm_t,
\eq
Here $u^\pm$ are the velocity fields, $p^\pm$ are the pressure fields, and we have normalized  the gravity  constant to $1$ for notational simplicity.  The fluids are assumed  to be quiescent at great depths,
\bq
\lim_{y\to \pm \infty}u^\pm(x, y)=0.
\eq
The normal velocity is continuous across the interface, i.e.
\bq\label{u.n:pm}
u^+\cdot n=u^-\cdot n\quad\text{on}~\Sigma_t\quad\text{on } \Sigma_t,
\eq
where $n=\frac{1}{\sqrt{1+|\nabla \eta|^2}}(-\nabla \eta, 1)$ is the upward pointing unit normal to $\Sigma_t$. Then the interface moves with the fluids:
\bq\label{kbc:pm}
\p_t\eta=\sqrt{1+|\nabla \eta|^2}u^-\cdot n\vert_{\Sigma_t}.
\eq
Neglecting the effect of surface tension, the pressure is continuous at the interface, i.e.
\bq\label{p:pm}
p^+=p^-\quad\text{on}~\Sigma_t.
\eq
We shall refer to the system \eqref{Omega+}-\eqref{p:pm} as the two-phase Muskat problem. When the top phase corresponds to vacuum, i.e. $\mu^+=\rho^+=0$, the two-phase Muskat problem reduces to the one-phase Muskat problem, in which case \eqref{p:pm} becomes
\bq
p^-=0\quad\text{on}~\Sigma_t.
\eq    
\subsection{Main result} 
A remarkable feature of the Muskat problem is that it can be recast as a {\it quasilinear} nonlocal parabolic equation for the free boundary $\eta$.  When the initial data is sufficiently smooth, local well-posedness for large data and global well-posedness for small data were established in \cite{DuchonRobert1984, Chen, EscherSimonett1997,  ConPug, CorGan09, Amb0, Amb, CorCorGan, CorCorGan2, CheGraShk}. It turns out that the problem is invariant with respect to the rescaling 
\[
\eta(x, t)\to \ld^{-1} \eta(\ld x, \ld t)\quad\forall \ld>0.
\]
Therefore, it is natural to develop well-posedness theories in function spaces that respect this scaling invariance. These include, among others, the  Lipschitz space $\dot W^{1, \infty}(\Rr^d)$, the Wiener algebra $\mathbb{A}^1(\Rr^d)$, and the homogeneous Besov space $\dot B^{1+\frac{d}{p}}_{p, q}(\Rr^d)$ for all $p$, $q\in [1, \infty]$. The latter includes the $L^2$-based Sobolev space $\dot H^{1+\frac{d}{2}}(\Rr^d)$ as a special case. 

There have been significant developments on well-posedness theories in or close to scaling invariant spaces.  Local well-posedness for large data arbitrarily close to scaling (``scaling$+\eps$'') is quite well understood. The case of no viscosity contrast ($\mu^+=\mu^-$) was treated in \cite{ConGanShvVic} with $W^{2, p}(\Rr)$ data for all $p\in (1, \infty]$, and in \cite{Matioc} with $H^{\tdm+\eps}(\Rr)$ data. The latter was then improved to the homogeneous space $\dot H^1(\Rr)\cap \dot H^{\tdm +}(\Rr)$ \cite{AL}, which is natural since the PDE annihilates constants. For {\it arbitrary viscosity contrast} and in {\it all dimensions}, local well-posedness was obtained in \cite{NguyenPausader} for $H^{1+\frac{d}{2}+\eps}(\Rr^d)$ data. See also \cite{AL} for the one-phase case. The method developed in \cite{NguyenPausader} could also handle the effect of surface tension \cite{Nguyen} and the zero surface tension limit \cite{FlynnNguyen} at the same regularity level. The only existing scaling invariant (modulo low-frequency assumptions) existence and uniqueness results are \cite{AlazardNguyen3, AlazardNguyen4} for the case of {\it no viscosity contrast}. They were obtained for data in $H^\tdm(\Rr)$ and $W^{1, \infty}(\Rr^2)\cap \dot H^2(\Rr^2)$, respectively.  

On the other hand, global well-posedness for small data is more  developed owing to the parabolic nature of the Muskat problem. The first result of this type is \cite{ConCorGanStr} which considers the equal viscosities case with data in the subcritical space $H^3(\Rr)$ but small (with explicit bound) in the Wiener norm 
\bq\label{def:Wiener}
\| \eta_0\|_{\mathbb{A}^1}=\int_\Rr |\xi||\widehat{f}(\xi)|d\xi.
\eq
This Wiener framework was then sucessfully employed to study a variety of problems such as the Navier-Stokes equations \cite{LeiLin}, the  Landau and non-cutoff Boltzmann equations \cite{DLSS},  the Peskin problem \cite{GJMS}. The results in \cite{ConCorGanStr} were extended to 3D in \cite{ConCorGanRodPStr}. Smallness in $\mathbb{A}^1$ implies smallness of the slope, i.e. of $\| \eta_0\|_{\dot W^{1, \infty}}$. Still for the case of no viscosity contrast,  if the initial Lipschitz norm is smaller than an explicit number and the profile grows sublinearly at infinity, then there exists a unique global classical solution \cite{Cam, Cam2}. Recent developments can address arbitrarily large slopes \cite{CordobaLazar, GancedoLazar, AlazardNguyen4} and even infinite slope \cite{AlazardNguyen3}. Precisely, \cite{CordobaLazar, GancedoLazar, AlazardNguyen4}  consider Lipschitz data that are small in the critical Sobolev space $\dot H^{1+\frac{d}{2}}(\Rr^d)$, $d\in \{1, 2\}$,  whereas the Lipschitz condition is removed in \cite{AlazardNguyen3} for $d=1$. As far as the general case of arbitrary viscosity contrast is concerned, to the best of our knowledge, the only scaling invariant result is \cite{GGJPS} for data in $L^2(\Rr^d)$ and small in $\mathbb{A}^1(\Rr^d)$, $d\in \{2, 3\}$. See also \cite{GG-JPS2019Bubble} for the global stability of bubbles. The case of viscosity contrast, including the one-phase problem, is more practical and mathematically challenging  since in addition to the evolution equation for the free boundary, one has to invert a nonlocal operator for the vorticity amplitude at the free boundary. Unlike the case of equal viscosities, there is no currently available $L^\infty$ or Lipschitz maximum principle when the viscosity contrast is nonzero. See \cite{GGJPS} for the equations and further discussions. For the one-phase problem, by taking advantage of  the maximum principle for the slope, it is proved in \cite{HGN} that for any periodic Lipschitz initial data, there exists a global strong solution  which is also the unique viscosity solution.

The purpose of this paper is to provide an alternative scaling invariant framework for the general Muskat problem: the Besov framework through the Dirichlet-Neumann operator.  To ensure the finiteness of the slope as in the Wiener framework, we consider the scaling invariant Besov spaces $\dot B^{\frac{d}{p}+1}_{p, 1}(\Rr^d)$, $p\in [1, \infty]$. The {\it largest} space in this class is $\dot B^1_{\infty, 1}(\Rr^d)$ which {\it strictly contains} $\mathbb{A}^1(\Rr^d)$ and 
\bq\label{intro:compare}
 \| \cdot\|_{\B^1}\le C\| \cdot\|_{\mathbb{A}^1}.
\eq
For example,  the function $f(x)=\cos(\sqrt{2}x)+\cos x$ belongs to $\dot B^1_{\infty, 1}(\Rr)$ but not to $\mathbb{A}^1$. Note that $f$ is neither periodic nor decaying at infinity. 

We establish global well-posedness for small data in $\B^1$  by exploiting a {\it semilinear-like structure} of the problem for such data. Our starting point is the reformulation \cite{NguyenPausader} of the general Muskat problem in terms of the Dirichlet-Neumann operators $G^\pm(\eta)$ associated to the fluid domains $\Omega^\pm$.  Precisely,  for a given function $f$, if $\phi^\pm$ solve 
\bq\label{potential:phi}
\begin{cases}
\Delta_{x, y} \phi^\pm=0\quad\text{in}~\Omega^\pm,\\
\phi^\pm=f\quad\text{on}~\Sigma,\\
\na_{x, y}\phi\to 0\quad\text{as } y\to \pm\infty,
\end{cases}
\eq
then 
\bq
G(\eta)^\pm f:=\sqrt{1+|\nabla \eta|^2}\frac{\p\phi^\pm}{\p n}.
\eq
\begin{prop}[\protect{\cite{NguyenPausader}}]\label{reform} Let $d\ge 1$.\\
1. $(u, p, \eta)$ solve the one-phase Muskat problem if and only if $\eta:\Rr^d\to \Rr$  obeys the equation 
\bq
\p_t\eta+\frac{\rho^-}{\mu^-} G^{-}(\eta)\eta=0.\label{eq:eta}
\eq
2.  $(u^\pm, p^\pm, \eta)$ solve the two-phase Muskat problem if and only if
\bq\label{eq:eta2p}
\p_t\eta=-\frac{1}{\mu^-}G^{-}(\eta)f^-,
\eq
where $f^\pm:=p^\pm\vert_{\Sigma}+\rho^\pm \eta$ satisfy
\bq\label{system:fpm}
\begin{cases}
 f^+-f^-= (\rho^+-\rho^-)\eta,\\
\frac{1}{\mu^+}G^+(\eta)f^+-\frac{1}{\mu^-}G^-(\eta)f^-=0.
\end{cases}
\eq
\end{prop}
We denote throughout this paper 
\[
 \kappa=\frac{\lb \rho\rb}{\mu^++\mu^-},\quad \lb \rho\rb=\rho^--\rho^+.
\]
It will be shown that $\ka$ is the dissipation coefficient. 

 Our main result is the following global well-posedness in Hadamard's sense for small data in $\B^1(\Rr^d)$.
 \begin{theo}\label{theo:global}
Consider either the one-phase  or the  two-phase Muskat problem in the stable case $\rho^+<\rho^-$. Let $\eta_0\in  \dot B^1_{\infty, 1}(\Rr^d)$, $d\ge 1$, be an initial datum. Then there exist positive constants $\delta$ and $C$, both depending only on $d$, such that the following holds: if $\| \eta_0\|_{\dot B^1_{\infty, 1}}\le \delta$ then for any $T>0$, \eqref{eq:eta} and \eqref{eq:eta2p}-\eqref{system:fpm} have a unique solution $\eta\in C([0, T]; \dot B^1_{\infty, 1}(\Rr^d))$ satisfying
\bq
\| \eta\|_{\wt L^\infty([0, T]; \dot B^1_{\infty, 1})}+\ka \| \eta\|_{\wt L^1([0, T]; \dot B^2_{\infty, 1})}\le C\| \eta_0\|_{\B^1}.
\eq
Here $\wt L^q([0, T]; \B^s)$ denotes the Chemin-Lerner space. Moreover, 
if $\eta^1$ and $\eta^2$ are two solutions as described above, then for all $T>0$,
\bq
\| \eta^1-\eta^2\|_{\wt L^\infty([0, T]; \dot B^1_{\infty, 1})}+\ka \| \eta^1-\eta^2\|_{\wt L^1([0, T]; \dot B^2_{\infty, 1})}\le C\| \eta^1(0)-\eta^2(0)\|_{\B^1}.
\eq 
\end{theo}
Theorem \ref{theo:global} does not require any  additional assumption on low frequencies of initial data. This is natural since for any constant $c$, $\eta+c$ is a solution if $\eta$ is a solution. This and the Lipschitz continuity of the solution map in $\B^1$ follow as a direct consequence of the fact that our solutions are constructed by means of a {\it fixed-point argument}. Since we do not rely on explicit contour equations \cite{CorGan09, GGJPS}, the proof of Theorem \ref{theo:global} works regardless of the dimension and {\it extends to the periodic setting}. We stress that the smallness $\delta$ is independent of all parameters of the problem. 
\begin{rema}
It would be interesting to investigate the long time decay and instant analyticity of the solutions constructed in Theorem \ref{theo:global}. However, in order to elucidate the main ideas, we do not pursue these issues in the current paper. 
\end{rema}
\subsection{The method of proof}
By virtue of Proposition \ref{reform}, the analysis of the Muskat problem reduces to that of the Dirichlet-Neumann operator. To fix ideas let us consider first the one-phase equation \eqref{eq:eta} in which the nonlinearity $G^-(\eta)\eta$ is fully nonlinear. In \cite{NguyenPausader},  a precise structure of $G^-(\eta)f$ is required  to handle large data. Precisely, we proved that for subcritical Sobolev regularity, $G^-(\eta)f$ can be decomposed into a first order elliptic (paradifferential) operator acting on $f$ plus a transport (paradifferential)  operator acting on $\eta$, and a lower order remainder. On the other hand, for small data, it turns out that the following different but simpler structure suffices
\bq\label{intro:expandDN}
G^-(\eta)f=|D_x|f+R^-(\eta)f,
\eq
where $R^-(\eta)f$ contains only quadratic and higher nonlinearities. The definition of $R^-(\eta)f$ involves $\eta$ and $v:\Rr^d\times \Rr_-\to \Rr$, where the latter is the harmonic potential $\phi$, defined by \eqref{potential:phi}, in the flattened domain $\{(x, y): y<0\}$. We shall prove that $v$ can be found as the unique fixed point in Chemin-Lerner spaces of some operator $\cT$ depending on $\eta$ (see \eqref{def:cT}), provided that $\eta$ is small in $\B^1$. Moreover, $R^-(\eta)f$ satisfies the good  Besov bound
\bq\label{intro:estR}
\begin{aligned}
\| R^-(\eta)f\|_{\dot B^1_{\infty, 1}}&\le C\|  \eta\|_{\dot B^1_{\infty, 1}}\|f\|_{\dot B^{2}_{\infty, 1}}+C\|  \eta\|_{\dot B^{2}_{\infty, 1}}\| f\|_{\dot B^1_{\infty, 1}}.
\end{aligned}
\eq
Interestingly, at the same $\B^1$ regularity, $R^-(\eta)f$ has  the contraction property 
\bq\label{intro:contractionR}
\begin{aligned}
  &\| [R^{-}(\eta_1)-R^{-}(\eta_2)]f\|_{\dot B^1_{\infty, 1}}\\
& \le  C\|(\eta_1, \eta_2)\|_{\dot B^{2}_{\infty, 1}}\| f\|_{\dot B^1_{\infty, 1}}+C\| \eta_\delta\|_{\dot B^1_{\infty, 1}}\|f\|_{\dot B^{2}_{\infty, 1}}+C\| \eta_\delta\|_{\dot B^{2}_{\infty, 1}}\|f\|_{\dot B^1_{\infty, 1}},\quad \eta_\delta=\eta_1-\eta_2.
 \end{aligned}
 \eq
Now in view of \eqref{intro:expandDN}, the one-phase equation \eqref{eq:eta} can be written in the Duhamel form 
\bq\label{intro:Duhamel}
\eta(t)=e^{-\ka t|D_x|}\eta_0-\ka \int_0^t e^{-\ka (t-\tau)|D_x|}(R^-(\eta)\eta)(\tau)d\tau.
\eq
This is reminiscent of mild solutions for the semilinear Navier-Stokes equations \cite{FujitaKato, CanMeyPlan} but  our nonlinearity $R^-(\eta)\eta$ is implicit and not purely quadratic. Nevertheless, by virtue of the boundedness \eqref{intro:estR} and the contraction \eqref{intro:contractionR} of $R^-(\eta)f$, a unique fixed point $\eta$ of \eqref{intro:Duhamel} can be easily obtained in the space  $\wt L^\infty([0, T]; \B^1)\cap \wt L^1([0, T]; \B^2)$ for any $T>0$, provided that $\| \eta_0\|_{\B^1}\le \delta(d)\ll 1$. In other words, for small data, the quasilinear Muskat problem can be {\it treated as a  semilinear equation}. The above fixed point argument allows us to bypass any additional assumption on low frequencies of data (for instance, $\eta_0\in L^2$  in \cite{GGJPS}) and prove at once the Lipschitz continuity of the solution map in the top topology of $\B^1$.

 Regarding the two-phase problem, owing to the linearization \eqref{intro:expandDN} we find that $f^-$, the solution of \eqref{system:fpm}, can be again obtained as the unique fixed point of some operator, provided that $\eta$ is small in $\B^1$. Moreover, $f^-$ has the same Besov regularity as $\eta$, so that the two-phase equation \eqref{eq:eta2p} can be analyzed exactly the same as the one-phase equation \eqref{eq:eta}. 

A review of homogeneous Besov spaces is given in Section \ref{section:prelim}. Section \ref{section:DN} is devoted to the linearization, boundedness and contraction properties of the Dirichlet-Neumann operator in Besov spaces. The proof of Theorem \ref{theo:global} is given in Section \ref{section:proof}. After the preliminaries in Section \ref{section:prelim}, the remainder of the paper is self-contained.
\section{A review of homogeneous Besov spaces}\label{section:prelim}
\subsection{Homogeneous Besov spaces}
\begin{prop}[\protect{\cite[Proposition 2.10]{BCD}}] Let $\mathcal{C}$ be the annulus $\{\xi\in \Rr^d: \frac 34\le |\xi|\le \frac 83\}$. There exists a  radial function $\psi$ valued in $[0, 1]$ and belonging  to $C_c^\infty(B(0, \frac 43))$  such that the following hold
\begin{align}
\label{sum2}&\sum_{j\in \Zz}\varphi(2^{-j}\xi)=1\quad\forall \xi\in \Rr^d\setminus\{0\},\\
&\supp \varphi(2^{-j}\cdot)\cap \supp \varphi(2^{-j'}\cdot)=\emptyset\quad\text{for}~|j-j'|\ge 2,\\
&\mez\le \sum_{j\in \Zz}\varphi^2(2^{-j}\xi)\le 1\quad\forall \xi\in \Rr^d\setminus\{0\}.
\end{align}
\end{prop}
\begin{defi}
The Littlewood-Paley dyadic block $\dot \Delta_j$ is defined by the Fourier multiplier  $\dot \Delta_ju=\varphi(2^{-j}D_x)u$. The low-frequency cut-off operator $\dot S_j$ is defined by
\bq
\dot S_j=\sum_{j'\le j-1}\dot \Delta_{j'}.
\eq
\end{defi}
Let $\cP(\Rr^d)$ and  $\mathcal{S}'(\Rr^d)$ respectively denote the set of polynomials on $\Rr^d$ and  the space of tempered distributions on $\Rr^d$. It is well-known that
\[
\text{I}=\sum_{j\in \Zz}\dot \Delta_j \quad \text{in } \mathcal{S}'(\Rr^d)/\cP(\Rr^d).
\]
\begin{defi}
1) For $(p, r)\in [1, \infty]^2$ and $s\in \Rr$, the homogeneous Besov space $\dot B^s_{p, r}(\Rr^d)$ is the space of  $u\in \mathcal{S}'(\Rr^d)/\cP(\Rr^d)$ such that the following norm is finite 
\[
\| u\|_{\dot B^s_{p, r}(\Rr^d)}=\| 2^{sj}\| \dot \Delta_j u\|_{L^p(\Rr^d)}\|_{\ell^r(\Zz)}.
\]
2) For $I\subset \Rr$, $(p, q, r)\in [1, \infty]^3$ and $s\in \Rr$, the Chemin-Lerner norm is defined by
\[
\| u\|_{\wt L^q(I; \dot B^s_{p, r}(\Rr^d))}=\| 2^{sj}\| \dot \Delta_j u\|_{L^q(I; L^p(\Rr^d))}\|_{\ell^r(\Zz)}.
\]
We then define the Chemin-Lerner space $\wt L^q(I; \dot B^s_{p, r}(\Rr^d))$ to be the space of tempered distributions $u\in \mathcal{S}'(\Rr^{d+1})/\cP(\Rr^{d})$ and  $\| u\|_{\wt L^q(I; \dot B^s_{p, r}(\Rr^d))}<\infty$. 
\end{defi}
\begin{rema}
The inequality \eqref{intro:compare} holds since for $u\in \mathbb{A}^1$ we have
\begin{align*}
\| u\|_{\B^1}&=\sum_{j\in \Zz}2^j\| \dot \Delta_j u\|_{L^\infty(\Rr^d)}\le \sum_{j\in \Zz}2^j\| \psi_j \widehat{u}\|_{L^1(\Rr^d)}\\&\le\int_{\Rr^d}\sum_{j\in \Zz}2^j\psi_j(\xi)|\widehat{u}(\xi)|d\xi \le \frac43\int_{\Rr^d\setminus\{0\}}\sum_{j\in \Zz}\psi_j(\xi)|\xi||\widehat{u}(\xi)|d\xi=\frac43\| u\|_{\mathbb{A}^1}.
\end{align*}
\end{rema}
For all $s\in \Rr$ and $(p, q)\in [1, \infty]^2$, $\dot B^s_{p, q}(\Rr^d)$ is a Banach space (see \cite{Triebel}, page 240). The following proposition can be proved analogously to Theorem 2.25 \cite{BCD}.
\begin{prop}\label{prop:CLspace}
For all $s\in \Rr$  and $(p, q, r)\in [1, \infty]^3$, $\wt L^q(I;  B^s_{p, r}(\Rr^d))$ is a Banach space and satisfies the Fatou property:  for any sequence $(u_n)$ bounded in $\wt L^q(I;  B^s_{p, r}(\Rr^d))$, there exist a subsequence $(u_{n_k})$ and $u\in \wt L^q(I;  B^s_{p, r}(\Rr^d))$ such that 
\bq
\lim_{k\to \infty} u_{n_k}=u\quad\text{in } \mathcal{S}'(\Rr^{d+1})/\cP(\Rr^{d})\quad\text{and}\quad \|u\|_{\wt L^q(I;  B^s_{p, r})}\le \liminf_{k\to \infty} \|u_{n_k}\|_{\wt L^q(I;  B^s_{p, r})}.
\eq
\end{prop}
By Minkowski's inequality, we have 
\bq\label{compareCL}
\| u\|_{\wt L^q(I; \dot B^s_{p, r})}\le \| u\|_{L^q(I; \dot B^s_{p, r})}~~\text{if}~~ q\le r,\quad \| u\|_{\wt L^q(I; \dot B^s_{p, r})}\ge \| u\|_{L^q(I; \dot B^s_{p, r})}~~\text{if}~~ q\ge r.
\eq
We shall need to following lemma to prove the continuity in time of solutions.
\begin{lemm}\label{lemm:trace} Let $p\in [1, \infty]$ and $s,~s'\in \Rr$ with $s\ge s'$. If $u\in \wt L^\infty([a, b]; \dot B^s_{p, 1}(\Rr^d))$ and $\p_zu\in  \wt L^1([a, b]; \dot B^{s'}_{p, 1}(\Rr^d))$, then $u\in C([a, b]; \dot B^s_{p, 1}(\Rr^d))$.
\end{lemm}
\begin{proof}
 We first note that if $v\in \wt L^q([a, b];  \dot B^{s_1}_{p, r})$ then $\dot S_jv\in \wt L^q([a, b];  \dot B^{s_2}_{p, r})$ for all $s_2\ge s_1$ and all $j\in \Zz$. Applying this we obtain 
\[
\dot S_ju\in \wt L^\infty([a, b];  \dot B^s_{p, 1}),\quad \p_z\dot S_ju=\dot S_j \p_zu \in  \wt L^1([a, b];  \dot B^{s}_{p, 1})=  L^1([a, b];  \dot B^{s}_{p, 1}),
\]
whence $\dot S_j u\in C([a, b];   \dot B^{s}_{p, 1})$.  It then suffices to prove that the sequence of $\dot B^s_{p, 1}$-valued functions $\dot S_j u$ converges to $u$ uniformly on $[0, T]$ as $j\to \infty$. Indeed, since $u-\dot S_ju=\sum_{k\ge j}\dot \Delta_j u$ we have 
\[
\| u-\dot S_ju\|_{L^\infty([a, b]; \dot B^s_{p, 1})}\le \sum_{k\ge j}\sum_{\ell\in \Zz}2^{sk}\|\dot\Delta_\ell \dot\Delta_k u\|_{L^\infty([a, b]; L^p)} \le C\sum_{k\ge j-3}2^{sk}\| \dot\Delta_k u\|_{L^\infty([a, b]; L^p)}
\]
which converges to $0$ as $j\to \infty$ because it is the tail of the  series convergent to $\| u\|_{\wt L^\infty([a, b]; \dot B^s_{p, 1})}$. 
\end{proof}
The following  product rules and nonlinear estimates shall be used frequently.
\begin{prop}[\protect{\cite[Corollary 2.54]{BCD}}]
For $I\subset \Rr$, $(p, r, q)\in [1, \infty]^3$ and $s>0$, we have
\begin{align}\label{pr}
&\Vert u_1 u_2 \Vert_{\dot B^s_{p, r}}\le C \Vert u_1\Vert_{L^\infty}\Vert u_2\Vert_{\dot B^s_{p, r}}+C\Vert u_2\Vert_{L^\infty}\Vert u_1\Vert_{\dot B^s_{p, r}}\\
\label{pr:CL}
&\Vert u_1 u_2 \Vert_{\wt L^q(I; \dot B^s_{p, r})}\le C \Vert u_1\Vert_{L^\infty(I; L^\infty)}\Vert u_2\Vert_{\wt L^q(I; \dot B^s_{p, r})}+C\Vert u_2\Vert_{L^\infty(I; L^\infty)}\Vert u_1\Vert_{\wt L^q(I; \dot B^s_{p, r})}.
\end{align}
\end{prop}
We note that when $s=0$, \eqref{pr} and \eqref{pr:CL} fail in general. In particular, $\B^0$ is not an algebra. 
\begin{theo}[\protect{\cite[Theorem 2.61]{BCD}}]\label{est:nonl}
  Consider~$F\in C^\infty(\Cc^N)$ such that~$F(0)=0$. Let $s>0$ and $(p, r, q)\in [1, \infty]^3$. 
If  
\[
U\in  \wt L^{q}(I; \dot B^{s}_{p, r}(\Rr^d)^N)\cap L^\infty(I; L^\infty(\Rr^d)^N),
\]
 then $F(U)$ belongs to the same space  and
\begin{equation}\label{est:F(u):CL}
\Vert F(U)\Vert_{\wt L^{q}(I; \dot B^{s}_{p, r})}\le \mathcal{F}\bigl(\Vert U\Vert_{L^\infty(I; L^\infty)}\bigr)\Vert U\Vert_{\wt L^{q}(I; \dot B^{s}_{p, r})}
\end{equation}
 for some nondecreasing function $\cF:\Rr^+\to\Rr^+$ independent of $U$. 
\end{theo}
\subsection{Fractional heat kernel in homogeneous Besov spaces}
\begin{lemm}[\protect{\cite[Lemma 2.4]{BCD}}]\label{lemm:fracheat}
There exist positive constants $c$ and $C$, both depending only on $d$, such that for all $p\in [1, \infty]$, $t>0$ and $k\in \Zz$ we have
\[
\| e^{-t|D_x|}\dot\Delta_ku\|_{L^p(\Rr^d)}\le Ce^{-ct2^k}\| \dot\Delta_k u\|_{L^p(\Rr^d)}.
\]
\end{lemm}
\begin{prop}\label{prop:heat}
Let $s\in \Rr$, $\nu>0$, and $(p, r, q_1, q_2)\in [1, \infty]^4$  such that $q_2\le q_1$. Let $I=[a, b]$ where $a\in \Rr\cup \{-\infty\}$ and $b\in \Rr$. Then there exist positive constants $C_1$ and $C_2$, both depending only on $d$, such that 
\begin{align}\label{Duhamel1}
&\| e^{-\nu z|D_x|}u(x)\|_{\wt L^{q_1}(I; \dot  B^{s+\frac{1}{q_1}}_{p, r})}\le  \frac{C_1}{\nu^{\frac{1}{q_1}}}\| u\|_{\dot B^s_{p, r}},\\\label{Duhamel2}
&\left\| \int_a^ze^{-\nu(z-z')|D_x|}f(x, z')dz'\right\|_{\wt L^{q_1}(I; \dot  B^{s+\frac{1}{q_1}}_{p, r})}\le  \frac{C_2}{\nu^{1+\frac{1}{q_1}-\frac{1}{q_2}}}\| f\|_{\wt L^{q_2}(I; \dot B^{s-1+\frac{1}{q_2}}_{p, r})}.
\end{align}
\end{prop}
\begin{proof}
We first prove \eqref{Duhamel2}. By virtue of Lemma \ref{lemm:fracheat},
\[
\| \dot\Delta_je^{-\nu z|D_x|}u\|_{L^p_x}\le C_1e^{-c\nu z 2^j}\|\dot \Delta_j u\|_{L^p_x},
\]
where $C_1=C_1(d)$. This implies 
\[
\| \dot\Delta_je^{-\nu z|D_x|}u\|_{L^{q_1}(I; L^p_x)}\le \frac{C_2}{\nu^{\frac{1}{q_1}}}2^{-\frac{j}{q_1}}\|\dot \Delta_j u\|_{L^p_x},
\]
where $I=[a, b]$ and $C_2=C_1(cp)^\frac{-1}{p}\le C_3=C_3(d)$. Consequently,
\[
\| e^{-\nu z|D_x|}u\|_{\wt L^{q_1}(I; \dot B^{s+\frac{1}{q_1}}_{p, r})}\le \frac{C_3}{\nu^{\frac{1}{q_1}}}\| u\|_{\dot B^{s}_{p, r}}
\]
which proves \eqref{Duhamel2}. 

As for \eqref{Duhamel1}, we apply Lemma \ref{lemm:fracheat} and Young's inequality in $z$ to have
\[
\begin{aligned}
\left\|\dot \Delta_j\int_a^ze^{-\nu(z-z')|D_x|}f(x, z')dz'\right\|_{L^{q_1}_z(I; L^p_x)}&=\left\|\|\int_a^ze^{-\nu(z-z')|D_x|}\dot \Delta_jf(\cdot, z')dz'\|_{L^p_x}\right\|_{L^{q_1}_z(I)}\\
&\le C_1\|\int_a^be^{-\nu c(z-y)2^j}\|\dot \Delta_jf(\cdot, z')\|_{L^p_x}dz'\|_{L^{q_1}_z(I)}\\
&\le \frac{C_4}{\nu^{1+\frac{1}{q_1}-\frac{1}{q_2}}}2^{j(-1+\frac{1}{q_2}-\frac{1}{q_1})}\|\Delta_jf\|_{L^{q_2}_z(I; L^p_x)},
\end{aligned}
\]
where $C_4=C_4(d)$. It follows that 
\[
\left\| \int_a^ze^{-(z-y)|D_x|}f(x, y)dy\right\|_{L^{q_1}_z(I; \dot B^{s+\frac{1}{q_1}}_{p, r})}\le  \frac{C_4}{\nu^{1+\frac{1}{q_1}-\frac{1}{q_2}}}\| f\|_{L^{q_2}_z(I; \dot B^{s-1+\frac{1}{q_2}}_{p, r})}
\]
which completes the proof of \eqref{Duhamel1}.
\end{proof}
\section{The Dirichlet-Neumann operator in $\dot B^r_{\infty, 1}$} \label{section:DN}
\subsection{Linearization and boundedness}\label{section:contDN}
We study the Dirichlet-Neumann operators associated to the fluid domains $\Omega^\pm$. The time variable is forgotten throughout. In $\Omega^-$, we consider the elliptic problem
\bq\label{eq:elliptic}
\begin{cases}
\Delta_{x, y} \phi=0\quad\text{in}~\Omega^-,\\
\phi=f\quad\text{on}~\Sigma,\\
\na_{x, y}\phi\to 0\quad\text{as}~y\to -\infty.
\end{cases}
\eq
The Dirichlet-Neuman operator associated to $\Omega^-$ is (formally) defined by
\bq\label{def:DN}
G^-(\eta) f=\sqrt{1+|\nabla \eta|^2}\frac{\p\phi}{\p n},
\eq
where we recall that $n$ is the upward-pointing unit normal to $\Sigma$. Similarly, if $\phi$ solves the elliptic problem \eqref{eq:elliptic} with $\Omega^-$ replaced by $\Omega^+$, then we also define $G^+(\eta)f$ by the right-hand side of \eqref{def:DN}. 
Note that $n$ is inward-pointing for $\Omega^+$, so that $G^+(\eta)$ is a nonpositive operator. We shall only state results for $G^-(\eta)$ since corresponding results for  $G^+(\eta)$ are completely parallel. 

Assume that $\phi$ is a smooth solution of \eqref{eq:elliptic}. We straighten the free boundary  using the change of variables $\Rr^d\times J\ni (x, z)\mapsto (x, \varrho(x, z))\in \Omega^-$, where $J=(-\infty, 0)$ and
\bq\label{def:rho}
\varrho(x, z)=z+H(x, z),\quad H(x, z)=e^{z|D_x|}\eta(x),\quad (x, z)\in \Rr^d\times J.
\eq
Clearly, $\varrho(x, 0)=\eta(x)$ and $\varrho(x,z)\to -\infty$ as $z\to -\infty$. We have $\p_z\varrho(x, z)=1+e^{z|D_x|}|D_x|\eta(x)$ and 
\[
\| e^{z|D_x|}|D_x|\eta\|_{L^\infty(\Rr^d\times J)}
\le C\| |D_x| \eta\|_{\dot B^0_{\infty, 1}}
\]
where we have used Lemma \ref{lemm:fracheat}. This implies

\begin{lemm}\label{lemm:diffeo} There exists a positive constant $c_0=c_0(d)\in (0, 1)$ such that if 
\bq\label{cd:diffeo}
\| |D_x| \eta\|_{\dot B^0_{\infty, 1}}\le c_0\quad\text{then}\quad \p_z\varrho\ge \mez
\eq
and thus the mapping $(x, z)\in \Rr^d\times J\mapsto \Omega^-$ is a Lipschitz diffeomorphism.
 \end{lemm}
  A direct calculation shows that if $g:\Omega^-\to \Rr$ then $\wt g(x, z)=g(x, \varrho(x, z))$  satisfies 
\bq\label{div:eq}
\cnx_{x,z}(\mathcal{A}\na_{x,z}\wt g)(x, z)=\p_z\varrho(\Delta_{x, y}g)(x, \varrho(x, z))
\eq
with 
\bq\label{def:matrixA}
\mathcal{A}=
\begin{bmatrix}
\p_z\varrho\text{I} & -\na_x\varrho\\
-(\na_x\varrho)^T& \frac{1+|\na_x\varrho|^2}{\p_z\varrho}
\end{bmatrix}
= \text{I}+
\begin{bmatrix}
(|D_x|H)\text{I}& -\na_xH\\
-(\na_xH)^T & \frac{|\na_xH|^2-|D_x|H}{1+|D_x|H}
\end{bmatrix}.
\eq
Since $\phi$ is harmonic in $\Omega^-$, $v(x, z)=\phi(x,\varrho(x, z))$ satisfies 
\bq\label{diveq:v}
\cnx_{x,z}(\mathcal{A}\na_{x,z}v)(x, z)=0\quad \text{in} ~\Rr^d\times J,
\eq
or equivalently, 
\bq\label{eq:Deltav}
\Delta_{x, z}v=\p_zQ_a[v]+|D_x|Q_b[v]\quad\text{in} ~\Rr^d\times J,
\eq
 where
\bq\label{Qab}
\begin{aligned}
&Q_a[v]=\na_xH\cdot \na_xv-\frac{|\na_xH|^2-|D_x|H}{1+|D_x|H}\p_zv,\\
&Q_b[v]=\mathcal{R}_x\big(\na_xH\p_zv-|D_x|H\na_xv\big),\quad\mathcal{R}_x=|D_x|^{-1}\cnx_x.
\end{aligned}
\eq
Then the Dirichlet-Neumann operator can be expressed in terms of $v$ as
\bq\label{DN:Hv}
\begin{aligned}
G(\eta)f&=\Big(\frac{1+|\na_x\varrho|^2}{\p_z\varrho}\p_zv-\na_x\varrho\cdot \na_xv\Big)\vert_{z=0}\\
&=\Big[\Big(1+\frac{|\na_xH|^2-|D_x|H}{1+|D_x|H}\Big)\p_zv-\na_xH\cdot \na_xv\Big]\vert_{z=0}\\
&=(\p_zv-Q_a[v])\vert_{z=0}.
\end{aligned}
\eq
By factorizing $\Delta_{x, z}v=(\p_z+|D_x|)(\p_z-|D_x|)v$ and setting 
\bq\label{def:w}
w=(\p_z-|D_x|)v-Q_a[v],
\eq
 we obtain from \eqref{eq:Deltav} that
\begin{align}\label{eq:w}
(\p_z+|D_x|)w=|D_x|(Q_b[v]-Q_a[v]).
\end{align}
In terms of $w$,  \eqref{DN:Hv} becomes
\bq\label{DN:w:0}
G^-(\eta)\eta=|D_x|f+w\vert_{z=0}.
\eq
Since $\na_{x, y}\phi\to 0$ as $y\to -\infty$, so are $\na_{x, z}v$, $Q_a[v]$ and $Q_b[v]$.  Then $v$ and $w$ as solutions of  \eqref{def:w} and \eqref{eq:w} are given by
\begin{align}\label{parabolic:v}
&v(x, z)=e^{z|D_x|}f(x)+\int_0^ze^{(z-z')|D_x|}\{w(x, z')+Q_a[v](x, z')\}dz', \quad z\le 0,\\\label{parabolic:w}
&w(x, r)=\int_{-\infty}^re^{-(r-\tau)|D_x|}|D_x|\{Q_b[v](x, \tau)-Q_a[v](x, \tau)\}d\tau,\quad r\le 0.
\end{align}
 Therefore, $v$ is a fixed point of the operator 
 \bq\label{def:cT}
 \begin{aligned}
 \mathcal{T}[v](x, z)&=e^{z|D_x|}f(x)+\int_0^ze^{(z-z')|D_x|}Q_a[v](x, z')dz'\\
 &+\int_0^ze^{(z-z')|D_x|}\int_{-\infty}^{z'}e^{-(z'-\tau)|D_x|}|D_x|\{Q_b[v](x, \tau)-Q_a[v](x, \tau)\}d\tau dz'.
 \end{aligned}
 \eq
 Note that $\cT$ depends on $\eta$. We shall prove that for small $\eta$, $\cT$ has a unique fixed point in suitable Chemin-Lerner spaces. Thus it is convenient to define
\begin{defi}
Let $\eta\in \dot B^1_{\infty, 1}$ satisfy $\| |D_x| \eta\|_{\dot B^0_{\infty, 1}}<c_0$. If $v$ is the fixed point of $\cT$, then the Dirichlet-Neumann operator $G(\eta)f$ is defined by
\bq\label{def:R}
G(\eta)f=|D_x|f+R^-(\eta)f:=|D_x|f+\int_{-\infty}^0e^{\tau |D_x|}|D_x|\{Q_b[v](x, \tau)-Q_a[v](x, \tau)\}d\tau.
\eq
\end{defi}
In order to obtain the existence of a unique fixed point of $\cT$, we shall appeal to the following elemetary lemma.\begin{lemm}\label{lemm:fp}
Let $E_1$ and $E_2$ be two norm spaces such that $E_1$ is complete. Assume that $E_2$ has the Fatou property: if $(u_n)$ is a bounded sequence in $E_2$ then there exist $u\in E_2$ and a subsequence $(u_{n_k})$ such that $u_{n_k}\to u$ in a sense weaker than norm convergence in $E_1$ and that 
\[
\| u\|_{E_2}\le C\liminf_{k\to \infty} \| u_{n_k}\|_{ E_2}.
\]
 Assume that $K:E_1\to E_1$ is a linear such that $K: E_1\cap E_2\to E_2$ and the following property holds.  There exist $(\alpha_1, \alpha_2)\in (0, 1)^2$ and $A>0$ such that for all $u\in E_1\cap E_2$, 
\begin{align}\label{boundE1}
&\| K(u)\|_{E_1}\le \alpha_1 \| u\|_{E_1},\\\label{boundE2}
&\|K(u)\|_{E_2}\le \alpha_2\| u\|_{E_2}+A \| u\|_{E_1}.
\end{align}
Then for any $u_0\in E_1\cap E_2$ there exists a unique fixed point $u_*\in E_1\cap E_2$ of the mapping 
$u\mapsto u_0+K(u)$.
\end{lemm}
\begin{proof}
By the Banach contraction  principle,  \eqref{boundE1} implies that $u_0+K(u)$ has a unique fixed point $u_*\in E_1$. Moreover, $u_*=\lim_{n\to \infty} u_n$ in $E_1$, where $u_{n+1}=u_0+K(u_n)$. For  $\beta_1>(1-\alpha_1)^{-1}$, using \eqref{boundE1} and induction  we find that $\|u_n\|_{E_1}< \beta_1 \|u_0\|_{E_1}$ for all $n$. Next, for
\[
\beta_2>\frac{\alpha_2}{1-\alpha_2},\quad M>\frac{A\beta_1}{1-\alpha_2},\quad R=\beta_2\| u_0\|_{E_2}+M\|u_0\|_{E_1},
\]
we claim that $\| u_n-u_0\|_{E_2}<R$ for all $n$. Indeed, this is obvious when $n=0$ and if it is true for $u_n$ then by \eqref{boundE2},
\[
\begin{aligned}
\| u_{n+1}-u_0\|_{E_2}&\le \alpha_2\| u_n\|_{E_2}+A\| u_n\|_{E_1}\\
&\le \alpha_2(R+\|u_0\|_{E_2})+A\beta_1\| u_0\|_{E_1}\\
&= \alpha_2(1+\beta_2)\| u_0\|_{E_2}+(\alpha_2M+A\beta_1)\| u_0\|_{E_1}\\
&\le\beta_2\| u_0\|_{E_2}+M\|u_0\|_{E_1}= R.
\end{aligned}
\]
Thus $(u_n)$ is bounded in $E_2$, so that the Fatou property imposed on $E_2$  guarantees $u_*\in E_1\cap E_2$. \end{proof}
Given~$I\subset \Rr\cup\{\pm \infty\}$, $\mu\in\Rr$ and $(p, q)\in [1, \infty]^2$, we define the interpolation space
\begin{equation}\label{def:X}
\begin{aligned}
X^\mu(I)&=\wt L^\infty(I;\dot B^\mu_{\infty, 1})\cap \wt L^1(I;\dot B^{\mu+1}_{\infty, 1}).
\end{aligned}
\end{equation}
Recall that the product rules \eqref{pr} and \eqref{pr:CL} fail in $\B^0$. As a substitute for $X^0(I)$ we define
\bq
X_*(I)=L^\infty(J; L^\infty)\cap \wt L^1(I; \B^1).
\eq
Note that 
\bq\label{trick:B0}
\|\cdot \|_{L^\infty(I; L^\infty)}\le \|\cdot\|_{\wt L^\infty(I; \B^0)},\quad \|\cdot \|_{X_*(I)}\le \|\cdot\|_{X^0(I)}.
\eq
\begin{prop}\label{lemm:elliptic}
1)  Let $\eta\in  \dot B^1_{\infty, 1}$ and $f\in \dot B^1_{\infty, 1}$. There exists  $c_1=c_1(d)<c_0<1$  such that if
 \bq\label{smallcd:1}
\| |D|_x\eta\|_{\dot B^0_{\infty, 1}}<c_1,
\eq
then $\cT$ has a unique fixed point $v$ in 
\bq
\mathcal{V}_*=\{v\in \mathcal{D}'(\Rr^d\times J): (|D_x|v, \p_zv)\in X_*(J)\}/\Rr.
\eq 
Moreover, there exists $C=C(d)$ such that 
\bq\label{V*est:v}
\| v\|_{\cV_*}\le C\||D_x| f\|_{\dot B^0_{\infty, 1}}
\eq
and
\bq\label{w0:0}
\| R^-(\eta)f\|_{\dot B^0_{\infty, 1}}\le C\|  |D_x|\eta\|_{\dot B^0_{\infty, 1}}\||D_x| f\|_{\dot B^{0}_{\infty, 1}}.
\eq
2)  Let $r>0$, $\eta\in  \dot B^1_{\infty, 1}\cap \dot B^{1+r}_{\infty, 1}$ and $f\in \dot B^1_{\infty, 1}\cap \dot B^{1+r}_{\infty, 1}$. There exists  $c_*=c_*(d,  r)<c_1$  such that if
 \bq\label{smallcd}
\| |D|_x\eta\|_{\dot B^0_{\infty, 1}}<c_*,
\eq
then the unique fixed point $v\in \cV_*$ of $\cT$ belongs to 
\bq
\mathcal{V}^r=\{v\in \mathcal{D}'(\Rr^d\times J): (|D_x|v, \p_zv)\in   X^r(J)\}/\Rr.
\eq 
Moreover, there exists $C=C(d, r)$ such that 
\bq\label{Xest:v}
\| v\|_{\cV^r}\le C\||D_x| f\|_{\dot B^{r}_{\infty, 1}}+C\|  |D_x|\eta\|_{\dot B^r_{\infty, 1}}\| |D_x|f\|_{\dot B^0_{\infty, 1}}
\eq
and 
\bq\label{w0}
\begin{aligned}
\| R^-(\eta)f\|_{\dot B^r_{\infty, 1}}&\le C\|  |D_x|\eta\|_{\dot B^0_{\infty, 1}}\||D_x| f\|_{\dot B^{r}_{\infty, 1}}+C\|  |D_x|\eta\|_{\dot B^r_{\infty, 1}}\| |D_x|f\|_{\dot B^0_{\infty, 1}}.
\end{aligned}
\eq
\end{prop}
\begin{proof}
For notational simplicity we shall write $Q_a=Q_a[v]$ and $Q_b=Q_b[v]$. With $w$ given by \eqref{parabolic:w}, we have 
\bq\label{def:opK}
\cT[v](x, z)=e^{z|D_x|}f(x)+K[u](x, z):=e^{z|D_x|}f(x)+\int_0^ze^{(z-z')|D_x|}\{w(x, z')+Q_a(x, z')\}dz'.
\eq
Throughout this proof, condition \eqref{smallcd:1} is always assumed, so that in many instances a nondecreasing function of $\| |D_x|\eta\|_{\B^0}$ is simply controlled by an absolute constant. 

{\bf 1.} Assume that $\eta,~f\in \B^1$. We shall prove that $\cT$ has a unique fixed point in $\cV_*$. Indeed, the estimate \eqref{Duhamel2} (with $\nu=1$)  yields
\bq\label{Xest:w10}
 \| w\|_{X^0(J)}\le C \| |D_x|(Q_b-Q_a)\|_{\wt L^1(J; \dot B^{0}_{\infty, 1})}\le C \| (Q_a, Q_b)\|_{\wt L^1(J; \dot B^{1}_{\infty, 1})}.
\eq
Upon changing the variables $z\mapsto -z$,    \eqref{Duhamel2} and \eqref{Xest:w10} imply
\bq\label{estdxK:0}
\begin{aligned}
\||D_x|(K[v])\|_{X^{0}(J)}&\le C\|  w\|_{\wt L^1(J; \dot B^{1}_{\infty, 1})}+C\|  Q_a\|_{\wt L^1(J; \dot B^{1}_{\infty, 1})}\le  C \| (Q_a, Q_b)\|_{\wt L^1(J; \dot B^{1}_{\infty, 1})}.
\end{aligned}
\eq
Now, since 
\bq\label{form:dzK}
\p_z K[u](x, z)=\int_0^ze^{(z-z')|D_x|}|D_x|\{w(x, z')+Q_a(x, z')\}dz'+w(x, z)+Q_a(x, z),
\eq
using \eqref{Duhamel2} and \eqref{Xest:w10} we estimate 
\bq\label{estdzK:01}
\|\p_z K[u](x, z)\|_{\wt L^1(J; \B^1)}\le C\| w\|_{\wt L^1(J; \B^1)}+C\| Q_a\|_{\wt L^1(J; \B^1)}\le C \| (Q_a, Q_b)\|_{\wt L^1(J; \dot B^{1}_{\infty, 1})}.
\eq
On the other hand, in view of \eqref{trick:B0}, \eqref{Duhamel2} and \eqref{Xest:w10},  we have
\bq\label{estdzK:00}
\begin{aligned}
\| \p_z K[u]\|_{L^\infty(J; L^\infty)}
&\le C\| w\|_{\wt L^1(J; \B^1)}+C\| Q_a\|_{\wt L^1(J; \B^1)}+\| w\|_{\wt L^\infty(J; \B^0)}+\| Q_a\|_{L^\infty(J; L^\infty)}\\
&\le C \| (Q_a, Q_b)\|_{\wt L^1(J; \dot B^{1}_{\infty, 1})}+\| Q_a\|_{L^\infty(J; L^\infty)}.
\end{aligned}
\eq
A combination of \eqref{estdxK:0}, \eqref{estdzK:01}, \eqref{estdzK:00} and \eqref{trick:B0} leads to
 \bq\label{estK:X*:5}
\|\big(|D_x|K[v], \p_zK[v]\big)\|_{X_*(J)}\le C\|(Q_a, Q_b)\|_{ \wt L^1(J; \dot B^{1}_{\infty, 1})}+\| Q_a\|_{L^\infty(J; L^\infty)}.
\eq
Note that if we choose $\cV^0$ in place of $\cV_*$, then we have to estimate $\p_z K[u]$ in $\wt L^\infty(J; \B^0)$. This in turn requires a bound for the nonlinear term $Q_a$ in $\wt L^\infty(J; \B^0)$. However, the product rule \eqref{pr:CL} does not hold when $s=0$. On the other hand, it is readily seen from the definition of $Q_a$ in \eqref{Qab} that 
\bq\label{estQ:Linfty}
\begin{aligned}
\| Q_a\|_{L^\infty(J; L^\infty)}&\le \| \na_x H\|_{L^\infty(J; L^\infty)}\| \na_xv\|_{L^\infty(J; L^\infty)}+C\|(|D_x|H, \na_xH)\|_{L^\infty(J; L^\infty)}\|\p_zv\|_{L^\infty(J; L^\infty)}\\
&\le C\| |D_x|\eta\|_{\B^0}\| \na_{x, z}v\|_{L^\infty(J; L^\infty)}.
\end{aligned}
\eq 
Consequently, in view of \eqref{estK:X*:5} and \eqref{estQ:Linfty}, we obtain
\bq\label{estK:X*:0}
\|\big(|D_x|(K[v]), \p_z(K[v])\big)\|_{X_*(J)}\le C\|(Q_a, Q_b)\|_{ \wt L^1(J; \dot B^{1}_{\infty, 1})}+C\| |D_x|\eta\|_{\B^0}\| \na_{x, z}v\|_{L^\infty(J; L^\infty)}.
\eq
For later use, we now estimate  $\|(Q_a, Q_b)\|_{ \wt L^1(J; \dot B^{1+r}_{\infty, 1})}$ for $r\ge 0$. Let us treat the most difficult term in $Q_a$
\[
I=\frac{|\na_xH|^2-|D_x|H}{1+|D_x|H}\p_zv
\]
which can be decomposed as
\[
I=|\na_xH|^2\p_zv-|\na_xH|^2F(|D_x|H)\p_zv-F(|D_x|H)\p_zv:=I_1+I_2+I_3,\quad F(z)=z/(1+z).
\]
Again, we only consider the most difficult term $I_2$. Since $\||D_x|H\|_{L^\infty}<c_1<1$, we can modify $F$ to a $C^\infty$ function taking same values on the range of $|D_x|H$. Let $r\ge 0$. Since $1+r>0$, using the tame estimate \eqref{pr:CL} gives
\[
\begin{aligned}
\|I_2\|_{\wt L^1(I; \dot B^{1+r}_{\infty, 1})}&\le C\left\| |\na_xH|^2F(|D_x|H)\right\|_{L^\infty(J; L^\infty)}\| \p_zv\|_{\wt L^1(J; \dot B^{1+r}_{\infty, 1})}\\
&\quad +C\left\| |\na_xH|^2F(|D_x|H)\right\|_{\wt L^1(J; \dot B^{1+r}_{\infty, 1})}\| \p_zv\|_{L^\infty(J; L^\infty)}.
\end{aligned}
\]
Since $|F(z)|\le 1$, we have
\[
\left\| |\na_xH|^2F(|D_x|H)\right\|_{L^\infty(J; L^\infty)}\le \| \na_xH\|^2_{L^\infty(J; L^\infty)}\le \| |D_x|\eta\|^2_{\dot B^0_{\infty, 1}}.
\]
Applying \eqref{pr:CL} again yields
\[
\begin{aligned}
\left\| |\na_xH|^2F(|D_x|H)\right\|_{\wt L^1(J; \dot B^{1+r}_{\infty, 1})}&\le C\| |\na_xH|^2\|_{L^\infty(J; L^\infty)}\left\|F(|D_x|H)\right\|_{\wt L^1(J; \dot B^{1+r}_{\infty, 1})}\\
&\quad +C\| |\na_xH|^2\|_{\wt L^1(J; \dot B^{1+r}_{\infty, 1})}\left\|F(|D_x|H)\right\|_{L^\infty(J; L^\infty)}\\
&\le  C\| \na_x\eta\|_{L^\infty}^2\left\|F(|D_x|H)\right\|_{\wt L^1(J; \dot B^{1+r}_{\infty, 1})}+C\| |\na_xH|^2\|_{\wt L^1(J; \dot B^{1+r}_{\infty, 1})}.
\end{aligned}
\]
To estimate $F(|D_x|H)$ we appeal to the nonlinear estimate \eqref{est:F(u):CL} to have
\[
\left\|F(|D_x|H)\right\|_{\wt L^1(J; \dot B^{1+r}_{\infty, 1})}\le C\||D_x|H\|_{\wt L^1(J; \dot B^{1+r}_{\infty, 1})}\le  C\||D_x|\eta\|_{\dot B^r_{\infty, 1}},\quad C=C(d, r).
\]
 By virtue of  \eqref{pr:CL} and \eqref{Duhamel1},
\[
\| |\na_xH|^2\|_{\wt L^1(J; \dot B^{1+r}_{\infty, 1})}\le C \| \na_x\eta\|_{L^\infty}\| \na_x\eta\|_{\dot B^{r}_{\infty, 1}}.
\]
We have proved that
\[
\left\| |\na_xH|^2F(|D_x|H)\right\|_{\wt L^1(J; \dot B^{1+r}_{\infty, 1})}\le C\||D_x|\eta\|_{ B^r_{\infty, 1}}.
\]
Putting together the above considerations leads to
\[
\| I_2\|_{\wt L^1(J; \dot B^{1+r}_{\infty, 1})}\le C\| |D_x|\eta\|_{\dot B^0_{\infty, 1}}\| \p_zv\|_{\wt L^1(J; \dot B^{1+r}_{\infty, 1})}+C\||D_x|\eta\|_{ B^r_{\infty, 1}}\| \p_zv\|_{L^\infty(J; L^\infty)},\quad C=C(d, r).
\]
Other terms in $Q_a$ and $Q_b$ can be estimated similarly, yielding 
\bq\label{estQ}
\begin{aligned}
\| (Q_a, Q_b)\|_{\wt L^1(J; \dot B^{1+r}_{\infty, 1})}\le  C\| |D_x|\eta\|_{\dot B^0_{\infty, 1}}\| \na_{x, z}v\|_{\wt L^1(J; \dot B^{1+r}_{\infty, 1})}+C\||D_x|\eta\|_{ B^r_{\infty, 1}}\| \na_{x, z}v\|_{L^\infty(J; L^\infty)}
\end{aligned}
\eq
for all $r\ge 0$, where $C=C(d, r)$. 

Applying \eqref{estQ} with $r=0$ then combining it with \eqref{estK:X*:0}, we arrive at a closed estimate for $K[v]$ in $\cV_*$
\bq\label{estK:V*}
\|  K[v]\|_{\cV_*}\le C\| |D_x|\eta\|_{\dot B^0_{\infty, 1}}\| v\|_{\cV_*},\quad C=C(d).
\eq
 We choose  $c_1<c_0$ sufficiently small such that $Cc_1<\frac 14$ and consider $\eta$ satisfying \eqref{smallcd:1}. Then by virtue of \eqref{estK:V*}, the Banach contraction principle yields the existence of a unique fixed point $v\in \cV_*$ of $\cT=e^{z|D_x|}f +K$. Since 
 \[
 \|e^{z|D_x|}f\|_{\cV_*}\le \|e^{z|D_x|}f\|_{\cV^0}\le C\| |D_x|f\|_{\dot B^0_{\infty, 1}},
 \]
  it follows from \eqref{estK:V*} that
\bq\label{V*est:v:1}
\|  v\|_{\cV_*}\le 2C\| |D_x|f\|_{\dot B^0_{\infty, 1}},\quad C=C(d).
\eq
Thus \eqref{V*est:v} holds. Finally, since $\| \cdot\|_{L^1\dot B^r_{\infty, 1}}=\| \cdot\|_{\wt L^1\dot B^r_{\infty, 1}}$, the left-hand side of \eqref{w0:0} is bounded by 
\bq\label{bound:w0}
\begin{aligned}
\| e^{\tau |D_x|}|D_x|(Q_b(x, \tau)-Q_a(x, \tau))\|_{L^1_\tau(J; \dot B^0_{\infty, 1})}
\le  C\| Q_b-Q_a\|_{\wt L^1_\tau(J; \dot B^{1}_{\infty, 1})},
\end{aligned}
\eq
so that  \eqref{w0:0} follows from \eqref{estQ} (with $r=0$) and \eqref{V*est:v}.

{\bf 2.} Assume that $\eta,~f\in \B^1\cap \B^{1+r}$ with $r>0$. As in  \eqref{Xest:w10} and \eqref{estdxK:0}, upon using the formula \eqref{form:dzK} for $\p_zK[v]$ we have
\bq\label{Xest:w1}
 \| w\|_{X^r(J)}\le  C \| (Q_a, Q_b)\|_{\wt L^1(J; \dot B^{1+r}_{\infty, 1})}
\eq
and
\bq\label{estdxK}
\begin{aligned}
\|(|D_x|K[v], \p_zK[v])\|_{X^r(J)}&\le C\|  w\|_{\wt L^1(J; \dot B^{1+r}_{\infty, 1})}+\|  Q_a\|_{\wt L^\infty(J; \dot B^{r}_{\infty, 1})}\\
&\le  C \| (Q_a, Q_b)\|_{\wt L^1(J; \dot B^{1+r}_{\infty, 1})}+\|  Q_a\|_{\wt L^\infty(J; \dot B^{r}_{\infty, 1})}.
\end{aligned}
\eq
 Unlike part 1), since $r>0$ the nonlinear term $Q_a$ can be handled using the product rule \eqref{pr:CL} and the nonlinear estimate \eqref{est:F(u):CL} as in the  proof of \eqref{estQ}. More precisely, 
 \bq\label{estQ:2}
 \begin{aligned}
\|  Q_a\|_{\wt L^\infty(J; \dot B^{r}_{\infty, 1})}\le C\| |D_x|\eta\|_{\dot B^0_{\infty, 1}}\| \na_{x, z}v\|_{\wt L^\infty(J; \dot B^{r}_{\infty, 1})}+C\||D_x|\eta\|_{ B^r_{\infty, 1}}\| \na_{x, z}v\|_{L^\infty(J; L^\infty)}.
 \end{aligned}
 \eq
In conjunction with \eqref{estQ} and \eqref{estQ:2}, \eqref{estdxK} implies that $K:\cV_*\cap \cV^r\to \cV^r$ and 
\bq\label{estv:Vr}
\|  K[v]\|_{\mathcal{V}^r}\le C\| |D_x|\eta\|_{\dot B^0_{\infty, 1}}\| v\|_{\mathcal{V}^r}+C\| |D_x|\eta\|_{\dot B^r_{\infty, 1}}\| v\|_{\cV_*},\quad C=C(d, r)
\eq
for all $v\in \cV_*\cap \cV^r$.  We choose  $c_*=c_*(d, r)<c_1$ sufficiently small so that $Cc_*<\frac 14$ and consider $\eta$ satisfying \eqref{smallcd}. Then in view of \eqref{estK:V*} and \eqref{estv:Vr}, $K$ satisfies the conditions \eqref{boundE1} and \eqref{boundE2} of Lemma \ref{lemm:fp} with $E_1=\cV_*$ and $E_2=\cV^r$.  Therefore, the unique fixed point $v\in \cV_*$ of $\cT$ belongs to $\cV^r$. Inserting \eqref{V*est:v:1} into \eqref{estv:Vr} we find
\[
\|  v\|_{\mathcal{V}^r}\le C| |D_x|f\|_{\dot B^{r}_{\infty, 1}}+C\|  |D_x|\eta\|_{\dot B^r_{\infty, 1}}\| |D_x|f\|_{\dot B^0_{\infty, 1}}, \quad C=C(d, r),
\]
whence \eqref{Xest:v} follows. We have used the fact that $\| e^{z|D_x|}f\|_{\cV^r}\le C(d)\| |D_x|f\|_{\B^r}$ which in turn follows from \eqref{Duhamel1}. Finally,  \eqref{w0} can be obtained as in \eqref{w0:0}.
 \end{proof}
Recalling the linearization $G^-(\eta)f=|D_x|f+R^-(\eta)f$, we deduce the following from \eqref{w0:0} and \eqref{w0}.
\begin{coro}\label{paraDN:infd}
Let $r\ge 0$, $\eta\in  \dot B^1_{\infty, 1}\cap \dot B^{1+r}_{\infty, 1}$ and $f\in \dot B^1_{\infty, 1}\cap \dot B^{1+r}_{\infty, 1}$. Assume that $\eta$ satisfies \eqref{smallcd}. Then  there exists $C=C(d, r)>0$ such that
\bq\label{est:G}
\|G^-(\eta)f\|_{\dot B^r_{\infty, 1}}\le  C\| |D_x|f\|_{\dot B^{r}_{\infty, 1}}+ C\|  |D_x|\eta\|_{\dot B^r_{\infty, 1}}\| |D_x|f\|_{\dot B^0_{\infty, 1}}. 
\eq 
\end{coro}
\subsection{Contraction}
We prove  contraction estimates for the Dirichlet-Neumann operator
\[
G^{-}(\eta_1)-G^{-}(\eta_2)=R^{-}(\eta_1)-R^{-}(\eta_2).
\]
\begin{prop}\label{prop:contraction}
 Let $r\ge 0$, $f \in \dot B^1_{\infty, 1}\cap\dot B^{r+1}_{\infty, 1}$ and $\eta_j\in  \dot B^1_{\infty, 1}\cap\dot B^{r+1}_{\infty, 1}$, $j=1, 2$. If $\eta_1$ and  $\eta_2$ satisfy \eqref{smallcd}, then there exists $C=C(d, r)$ such that
\bq\label{est:contraDN}
\begin{aligned}
  &\| [R^{-}(\eta_1)-R^{-}(\eta_2)]f\|_{\dot B^r_{\infty, 1}}\\
& \le  C\mathcal{A}_r\| |D_x|\eta_\delta\|_{\dot B^0_{\infty, 1}}\| |D_x|f\|_{\dot B^0_{\infty, 1}}+C\| |D_x|\eta_\delta\|_{\dot B^0_{\infty, 1}}\||D_x|f\|_{\dot B^{r}_{\infty, 1}}+C\| |D_x|\eta_\delta\|_{\dot B^r_{\infty, 1}}\||D_x|f\|_{\dot B^0_{\infty, 1}},
 \end{aligned}
 \eq
 where $\eta_\delta=\eta_1-\eta_2$ and 
 \bq\label{def:As}
 \mathcal{A}_s=\||D_x|\eta_1\|_{\dot B^s_{\infty, 1}}+\||D_x|\eta_2\|_{\dot B^s_{\infty, 1}}.
 \eq
\end{prop}
\begin{proof}
We recall from \eqref{def:R} that 
\[
R^-(\eta_j)f_j(x)=\int_{-\infty}^0e^{\tau |D_x|}|D_x|\{Q_{b, j}[v_j](x, \tau)-Q_{a, j}[v_j](x, \tau)\}d\tau,
\]
where $v_j$ is the unique fixed point of $\cT_j$ associated to $\eta_j$. Here we use the same notation as in subsection \ref{section:contDN} but with the index $j\in\{1, 2\}$. We use in addition the notation $g_\delta=g_1-g_2$. 

{\it Step 1.} Since $\| \cdot\|_{L^1\dot B^r_{\infty, 1}}=\| \cdot\|_{\wt L^1\dot B^r_{\infty, 1}}$, we deduce 
\[
A_r:=\| R^-(\eta_1)f_1-R^-(\eta_2)f_2\|_{\dot B^r_{\infty, 1}}\le C \| (Q_{b}[v])_\delta(x, \tau)-(Q_{a}[v])_\delta\|_{\wt L^1(J; \dot B^{r+1}_{\infty, 1})}.
\]
Let us consider the following typical term under the norm on the right-hand side
\[
T=\na_xH_1\cdot \na_xv_1-\na_xH_2\cdot \na_xv_2=\na_xH_1\cdot \na_xv_\delta+\na_xH_\delta\cdot \na_xv_2.
\]
 By the tame estimate \eqref{pr}, for $r\ge 0$ we have
\[
\begin{aligned}
\| T\|_{ \wt L^1(J; \dot B^{1+r}_{\infty, 1})}&\le C\|\na_xH_1\|_{L^\infty(J; L^\infty)}\|\na_xv_\delta\|_{\wt L^1(J; \dot B^{1+r}_{\infty, 1})}+ C\|\na_xH_1\|_{\wt L^1(J; \dot B^{r+1}_{\infty, 1})}\|\na_xv_\delta\|_{L^\infty(J; L^\infty)}\\
&\quad +C \|\na_xH_\delta\|_{L^\infty(J; L^\infty)}\|\na_xv_2\|_{\wt L^1(J; \dot B^{1+r}_{\infty, 1})}+ C\|\na_xH_\delta\|_{\wt L^1(J; \dot B^{r+1}_{\infty, 1})}\|\na_xv_2\|_{L^\infty(J; L^\infty)}.
\end{aligned}
\]
When $r=0$, this yields
\bq\label{estT:0}
\begin{aligned}
\| T\|_{ \wt L^1(J; \dot B^{1}_{\infty, 1})}&\le C\mathcal{A}_0\| v_\delta\|_{\cV_*}+C\| |D_x|\eta_\delta\|_{\dot B^0_{\infty, 1}}\| v_2\|_{\cV_*}\\
&\le C\big(\mathcal{A}_0\| v_\delta\|_{\cV_*}+C\| |D_x|\eta_\delta\|_{\dot B^0_{\infty, 1}}\| |D_x|f\|_{\B^0}\big):=E_*,
\end{aligned}
\eq
where we have used \eqref{V*est:v} for $v_2$. On the other hand, when $r>0$ we obtain
\bq\label{estT}
\begin{aligned}
&\| T\|_{ \wt L^1(J; \dot B^{1+r}_{\infty, 1})}\\
&\le C\Big\{\mathcal{A}_0\| v_\delta\|_{\cV^r}+\mathcal{A}_r\|v_\delta\|_{\cV_*}+\| |D_x|\eta_\delta\|_{\dot B^0_{\infty, 1}}\||D_x| f\|_{\dot B^{r}_{\infty, 1}}+\mathcal{A}_r\| |D_x|\eta_\delta\|_{\dot B^0_{\infty, 1}}\||D_x| f\|_{\dot B^0_{\infty, 1}}\\
&\quad+\| |D_x|\eta_\delta\|_{\dot B^r_{\infty, 1}}\||D_x| f\|_{\dot B^0_{\infty, 1}}\Big\} :=E_r,
\end{aligned}
\eq
where we have used \eqref{V*est:v} and \eqref{Xest:v} for $v_2$. We have proved that
\bq\label{estv:diff10}
\begin{aligned}
A_r\le C\| (Q_{b}[v])_\delta(x, \tau)-(Q_{a}[v])_\delta\|_{\wt L^1(J; \dot B^{r+1}_{\infty, 1})}\le \begin{cases} E_*\quad\text{when } r=0,\\ E_r\quad\text{when } r>0.\end{cases}
\end{aligned}
\eq 
{\it Step 2.} In view of \eqref{def:opK}, we have 
\[
v_\delta=(\cK[v])_\delta=\int_0^ze^{(z-z')|D_x|}\{w_\delta(x, z')+(Q_a[v])_\delta(x, z')\}dz'.
\]
The proof of \eqref{estK:X*:5} also gives 
\[
\| v_\delta\|_{\cV_*}\le C\|((Q_a[v])_\delta, (Q_b[v])_\delta)\|_{ \wt L^1(J; \dot B^{1}_{\infty, 1})}+\| (Q_a[v])_\delta\|_{L^\infty(J; L^\infty)}.
\]
Using the definition of $Q_a[v]$ we estimate
\[
\| (Q_a[v])_\delta\|_{L^\infty(J; L^\infty)}\le C\| |D_x|\eta_\delta\|_{\B^0}\| \na_{x, z}v_1\|_{L^\infty(J; L^\infty)}+C\| |D_x|\eta_2\|_{\B^0}\| \na_{x, z}v_\delta\|_{L^\infty(J; L^\infty)}\le E_*.
\]
Then upon recalling \eqref{estv:diff10}, we deduce from the two preceding estimates that
\[
\| v_\delta\|_{\cV_*}\le E_*=C\mathcal{A}_0\| v_\delta\|_{\cV^0}+C\| |D_x|\eta_\delta\|_{\dot B^0_{\infty, 1}}\| |D_x|f\|_{\B^0}.
\]
For $\eta_j$ satisfying \eqref{smallcd} with $c_*$ smaller if necessary so that $C\mathcal{A}_0\le 2Cc_*<1$, $\| v_\delta\|_{\cV^0}$ on the right-hand side can be absorbed by the right-hand side, giving
\bq\label{est:vdelta:V*}
\|  v_\delta\|_{\cV_*}\le C\| |D_x|\eta_\delta\|_{\dot B^0_{\infty, 1}}\| |D_x|f\|_{\dot B^0_{\infty, 1}}.
\eq
Plugging this back in $E_*$, we conclude in view of \eqref{estv:diff10} that 
\[
A_0\le C\| |D_x|\eta_\delta\|_{\dot B^0_{\infty, 1}}\| |D_x|f\|_{\dot B^0_{\infty, 1}}
\]
which is the estimate \eqref{est:contraDN} for $r=0$.

Next, let $r>0$. Analogously to \eqref{estdxK} we have
\bq
\| v_\delta\|_{\cV^r}\le C \| ((Q_a[v])_\delta, (Q_b[v])_\delta)\|_{\wt L^1(J; \dot B^{1+r}_{\infty, 1})}+\|  (Q_a[v])_\delta\|_{\wt L^\infty(J; \dot B^{r}_{\infty, 1})},
\eq
where
\begin{align*}
\|  (Q_a[v])_\delta\|_{\wt L^\infty(J; \dot B^{r}_{\infty, 1})}&\le C\| |D_x|\eta_\delta\|_{\B^r}\| \na_{x, z}v_1\|_{L^\infty(J; L^\infty)}+C\| |D_x|\eta_2\|_{\B^0}\| \na_{x, z}v_\delta\|_{\wt L^\infty(J; \B^r)}\\
&\le E_r.
\end{align*}
Therefore, it follows in view of \eqref{estv:diff10} that
\bq\label{vdelta:Vr}
\| v_\delta\|_{\cV^r}\le E_r.
\eq
Inserting \eqref{est:vdelta:V*} into $E_r$ on the right-hand side of \eqref{vdelta:Vr}, then absorbing $\| v_\delta\|_{\cV^r}$, we find 
\begin{align*}
\| v_\delta\|_{\cV^r} &\le C\mathcal{A}_r\| |D_x|\eta_\delta\|_{\dot B^0_{\infty, 1}}\| |D_x|f\|_{\dot B^0_{\infty, 1}}+C\| |D_x|\eta_\delta\|_{\dot B^0_{\infty, 1}}\||D_x|f\|_{\dot B^{r}_{\infty, 1}}\\
&\quad+C\| |D_x|\eta_\delta\|_{\dot B^r_{\infty, 1}}\||D_x|f\|_{\dot B^0_{\infty, 1}}.
\end{align*}
Using this bound for  $\| v_\delta\|_{\cV^r}$ in $E_r$, we deduce that $E_r$ obeys the same bound as $\| v_\delta\|_{\cV^r}$. Finally, \eqref{est:contraDN} follows from \eqref{estv:diff10}.
\end{proof}
\section{Proof of Theorem \ref{theo:global}}\label{section:proof}
We first prove a simple fixed-point lemma tailored to our problem.
\begin{lemm}\label{lemm:fixedpoint}
Let $(E, \| \cdot\|)$ be a Banach space, and let $\nu>0$. Denote by $B_\nu$ the closed ball of radius $\nu$ centered at $0$ in $E$.  Assume that $\mathcal{B}:E\times E\to E$ and there exists $\cF:\Rr^+\to \Rr^+$ such that
\bq\label{cB:4}
\nu\cF(2\nu)\le \mez
\eq
and the following conditions hold.
\begin{itemize}
\item For all $x\in B_\nu$, $\mathcal{B}(x, \cdot)$ is linear and 
\bq\label{cB:1}
\| \mathcal{B}(x, y)\|\le \cF(\| x\|)\|x\|\| y\|\quad\forall y\in B_\nu.
\eq
\item For all $x_1,~x_2,~y\in B_\nu$, 
\bq\label{cB:2}
\| \mathcal{B}(x_1, y)-B(x_2, y)\|\le \cF(\| x_1\|+\|x_2\|)\| x_1-x_2\|\|y\|.
\eq
\end{itemize}
Then, there exists $\delta=\delta(\nu, \cF)>0$ small enough such for all $x_0\in E$ with norm less than $\delta$, $x\mapsto x_0+\mathcal{B}(x, x)$ has a unique fixed point $x_*$ in  $B_\nu$ with $\|x_*\|\le 2\|x_0\|$.
\end{lemm}
\begin{proof}
For all $x,~y\in B_\nu$, combining \eqref{cB:1} and \eqref{cB:2} gives
\bq\label{cB:3}
\begin{aligned}
\| \mathcal{B}(x, x)-\mathcal{B}(y, y)\|&\le \| \mathcal{B}(x, x-y)\|+\|\mathcal{B}(x, y)-\mathcal{B}(y, y)\|\\
&\le \cF(\| x\|)\|x\|\|x-y\|+\cF(\| x\|+\|y\|)\|x-y\|\|y\|.
\end{aligned}
\eq
Let $\delta\in(0, \frac{\nu}{2})$ be such that $2\delta\cF(2\delta)< 1$. For $\| x_0\|<\delta$, we construct the fixed point of $x_0+\mathcal{B}(x, x)$ by the iterative scheme $x_{n+1}=x_0+\mathcal{B}(x_n, x_n)$, $n\ge 0$. Using \eqref{cB:1} and induction we find that $\| x_n\|< 2\| x_0\| <\nu$ for all $n\ge 0$. Then, \eqref{cB:3} implies
\begin{align*}
\| x_{n+1}-x_n\|&=\| \mathcal{B}(x_n, x_n)-\mathcal{B}(x_{n-1}, x_{n-1})\|\\
&\le 2\delta\cF(2\delta)\|x_n-x_{n-1}\|+2\delta\cF(4\delta)\|x_n-x_{n_1}\|.
\end{align*}
Choosing $\delta$ smaller so that $2\delta\cF(2\delta)+2\delta\cF(4\delta)\le \mez$, we obtain $\| x_{n+1}-x_n\|\le \mez\|x_n-x_{n_1}\|$. Thus the sequence $(x_n)$ converges to $x_*\in E$. In particular, $\| x_*\|\le 2\| x_0\|<\nu$ and it follows from \eqref{cB:3} that  $\|\mathcal{B}(x_n, x_n)-\mathcal{B}(x_*, x_*)\|\to 0$. Therefore, letting $n\to \infty$ in the iterative scheme yields $x_*=x_0+\mathcal{B}(x_*, x_*)$. If $x'$ is another fixed point in $B_\nu$ , then by \eqref{cB:3} and \eqref{cB:4},
\[
\| x_*-x'\|=\|\mathcal{B}(x_*, x_*)-\mathcal{B}(x', x')\|\le 2\delta\cF(2\delta)\|x_*-x'\|+\nu\cF(2\nu)\|x-y\|\le \frac{3}{4}\| x_*-x'\|,
\]
provided that $ 2\delta\cF(2\delta)\le \frac{1}{4}$. Therefore, $x_*$ is the unique fixed point in $B_\nu$.
\end{proof}
\subsection{The one-phase problem}
From \eqref{eq:eta} and \eqref{def:R}, the one-phase Muskat problem can be written as the nonlinear fractional heat equation
\bq\label{PDE:onephase}
\p_t\eta+\ka |D_x|\eta=-\ka R^-(\eta)\eta,\quad \eta\vert_{t=0}=\eta_0,
\eq
or equivalently
\bq\label{sol:eta}
\eta(t)=e^{-\ka t|D_x|}\eta_0-\mathcal{B}(\eta, \eta),
\eq
where
\[
\mathcal{B}(\eta, f)(x, t)=\ka \int_0^t e^{-\ka (t-\tau)|D_x|}(R^-(\eta)f)(\tau)d\tau.
\]
Let $X^1_\ka([0, T])$ denote the Banach space $X^1([0, T])$ (see \eqref{def:X}) endowed with the equivalent norm
\bq
\|\eta\|_{X^1_\ka([0, T])}=\| \eta\|_{\wt L^\infty([0, T]; \B^1)}+\ka\| \eta\|_{\wt L^1([0, T]; \B^2)}.
\eq
 By virtue of Lemma \ref{lemm:trace}, $X^1_k([0, T])\subset C([0, T]; \B^1)$. Then the global well-posedness statement in Theorem \ref{theo:global} is a consequence of the following lemma.
\begin{lemm}\label{proof:onephase}
 There exists a small number $\delta>0$ such that if  $\| \eta_0\|_{\B^1}<\delta$, then $e^{-\ka t|D_x|}\eta_0-\mathcal{B}(\eta, \eta)$ has a unique fixed point $\eta$ in $X^1_k([0, T])$ for any $T>0$ with norm less than $C\| \eta_0\|_{\B^1}$, $C=C(d)$.
 \end{lemm}
\begin{proof}
Let $T>0$ be arbitrary.  We appeal to Lemma \ref{lemm:fixedpoint} with $E=X^1_\ka([0, T])$ and $\nu=c_*/2$, where $c_*$ is given in \eqref{smallcd}. Note that if $\| \eta\|_{X^1_\ka([0, T])}<c_*$, then $\| \eta\|_{C([0, T]; \B^1)}<c_*$, so that for all $t\in [0, T]$, $\eta(t, \cdot)$ satisfies the smallness condition \eqref{smallcd}. In addition, in view of \eqref{Duhamel1} we have
\[
\| e^{-\ka t|D_x|}\eta_0\|_{X^1_\ka([0, T])}\le C_1\| \eta_0\|_{\B^1},\quad C_1=C_1(d).
\]
Applying the parabolic estimate \eqref{Duhamel2}  with  $(q_1, q_2)=(\infty, 1)$ and $(q_1, q_2)=(1, 1)$, we obtain
\[
\| \mathcal{B}(\eta, f)\|_{X^1_\ka([0, T])}\le C\ka\|R^-(\eta)f\|_{\wt L^1([0, T]; \B^1)}=C\ka\|R^-(\eta)f\|_{L^1([0, T]; \B^1)}.
\]
Recall from \eqref{w0} that
\[
\|R^-(\eta)f\|_{\B^1}\le C\|  |D_x|\eta\|_{\dot B^0_{\infty, 1}}\||D_x| f\|_{\dot B^1_{\infty, 1}}+C\|  |D_x|\eta\|_{\dot B^1_{\infty, 1}}\| |D_x|f\|_{\dot B^0_{\infty, 1}}.
\]
It follows that 
\bq\label{check:cB1}
\begin{aligned}
\| \mathcal{B}(\eta, f)\|_{X^1_\ka([0, T])}& \le C\|  |D_x|\eta\|_{L^\infty_t\dot B^0_{\infty, 1}}\||D_x| f\|_{L^1_t\dot B^1_{\infty, 1}}+C\|  |D_x|\eta\|_{L^1_t\dot B^1_{\infty, 1}}\| |D_x|f\|_{L^\infty_t\dot B^0_{\infty, 1}}\\
&\le C\|\eta\|_{X^1_\ka([0, T])}\|f\|_{X^1_\ka([0, T])},\quad C=C(d),
\end{aligned}
\eq
where we have used \eqref{compareCL} in the second inequality. 

Similarly, we have
\[
\| \mathcal{B}(\eta_1, f)-\mathcal{B}(\eta_2, f)\|_{X^1_\ka([0, T])}\le C\ka\|R^-(\eta_1)f-R^-(\eta_2)f\|_{\wt L^1([0, T]; \B^1)}. 
\]
Recall the contraction estimate \eqref{est:contraDN}:
 \begin{align*}
 &\|R^-(\eta_1)f-R^-(\eta_2)f\|_{\B^1}\\
 &\le C\mathcal{A}_1\| |D_x|\eta_\delta\|_{\dot B^0_{\infty, 1}}\| |D_x|f\|_{\dot B^0_{\infty, 1}}+C\| |D_x|\eta_\delta\|_{\dot B^0_{\infty, 1}}\||D_x|f\|_{\dot B^{1}_{\infty, 1}}+C\| |D_x|\eta_\delta\|_{\dot B^1_{\infty, 1}}\||D_x|f\|_{\dot B^0_{\infty, 1}},
 \end{align*}
 where $C=C(d)$ and $ \mathcal{A}_1=\||D_x|\eta_1\|_{\dot B^1_{\infty, 1}}+\||D_x|\eta_2\|_{\dot B^1_{\infty, 1}}$. We deduce
 \bq\label{check:cB2}
 \begin{aligned}
 \| \mathcal{B}(\eta_1, f)-\mathcal{B}(\eta_2, f)\|_{X^1_\ka([0, T])}&\le C\Big(\| \eta_1\|_{X^1_\ka([0, T])}+ \|\eta_2\|_{X^1_\ka([0, T])}+2\Big)\| \eta_\delta\|_{X^1_\ka([0, T])}\| f\|_{X^1_\ka([0, T])}.
 \end{aligned}
 \eq
In view of \eqref{check:cB1} and \eqref{check:cB2}, $\mathcal{B}$ satisfies the conditions \eqref{cB:1} and \eqref{cB:2} in Lemma \ref{lemm:fixedpoint} with $\cF(z)=C(z+2)$. Finally, choosing $c_*$ smaller if necessary so that $\nu=c_*/2$ satisfies \eqref{cB:4}, we can apply Lemma \ref{lemm:fixedpoint} to conclude the proof.
\end{proof}
Let $\eta_0^j\in \B^1$, $j=1, 2$ be initial data with norm less than $\delta$ given in Lemma \ref{proof:onephase}. Denoting by $\eta^j$ the corresponding solutions. Then the estimates \eqref{check:cB1} and \eqref{check:cB2} imply 
\[
\begin{aligned}
\| \eta^1-\eta^2\|_{X^1_k([0, T])}&\le \| e^{-\ka t|D_z|}(\eta^1_0-\eta^2_0)\|_{X^1_\ka([0, T])}+\| \mathcal{B}(\eta^1, \eta^1)-\mathcal{B}(\eta^2, \eta^2)\|_{X^1_\ka([0, T])}\\
&\le C\|\eta^1_0-\eta^2_0\|_{\B^1}+C\|\eta^1\|_{X^1_\ka([0, T])}\| \eta^1-\eta^2\|_{X^1_\ka([0, T]}\\
&\quad+C\| \eta^1-\eta^2\|_{X^1_\ka([0, T])}\| \eta^2\|_{X^1_\ka([0, T])},
\end{aligned}
\]
where we have used that $\| \eta^j\|_{X^1_\ka([0, T])}\le 2\delta<1$. Therefore, choosing $\delta$ smaller if necessary we obtain the stability estimate 
\bq
\| \eta^1-\eta^2\|_{X^1_k([0, T])}\le C\| \eta^1_0-\eta^2_0\|_{\B^1}\quad\forall T>0,\quad C=C(d).
\eq 
This shows that the solution map is Lipschitz continuous in the topology of $\B^1$, thereby establishing global well-posedness  in the strong sense of Hadamard.
\subsection{The two-phase problem}
Let us recall the Dirichlet-Neumann reformulation of the two-phase Muskat problem:
\begin{align}\label{eq:eta2p2}
&\p_t\eta=-\frac{1}{\mu^-}G^{-}(\eta)f^-,\\
\label{system:fpm2}
&\begin{cases}
 f^+-f^-= -\lb \rho\rb\eta,\\
\frac{1}{\mu^+}G^+(\eta)f^+-\frac{1}{\mu^-}G^-(\eta)f^-=0.
\end{cases}
\end{align}
Substituting $f^+$ from the first equation in \eqref{system:fpm2}, the second equation becomes 
\[
\frac{1}{\mu^-}G^-(\eta)f^-=\frac{1}{\mu^+}G^+(\eta)f^--\frac{\lb \rho\rb}{\mu^+}G^+(\eta)\eta.
\]
Then setting 
\bq\label{def:Rpm}
R^\pm(\eta)g=G^\pm(\eta)g\mp  |D_x|g,
\eq
we obtain after arranging terms
\bq\label{eq:f-:0}
|D_x|f^-=\frac{1}{\mu^++\mu^-}(\mu^-R^+(\eta)-\mu^+R^-(\eta))f^--\frac{\lb \rho\rb \mu^-}{\mu^++\mu^-}G^+(\eta)\eta.
\eq
Plugging this into the linearization $G^-(\eta)f^-=|D_x|f^-+R^-(\eta)f^-$, it follows in view of \eqref{eq:eta2p2} that
\[
\begin{aligned}
\p_t \eta
&=-\frac{1}{\mu^++\mu^-}(R^+(\eta)+R^-(\eta))f^--\frac{\lb \rho\rb}{\mu^++\mu^-}|D_x|\eta+\frac{\lb \rho\rb}{\mu^++\mu^-}R^+(\eta)\eta.
\end{aligned}
\]
In the stable case $\rho^+<\rho^-$, $\ka=\frac{\lb \rho\rb}{\mu^++\mu^-}>0$ and $\eta$ obeys the nonlinear fractional heat equation
\bq\label{PDE:twophase}
\p_t \eta+\ka|D_x|\eta=-\ka (R^+(\eta)+R^-(\eta))\frac{f^-}{\lb \rho\rb}+\ka R^+(\eta)\eta.
\eq
In order for \eqref{PDE:twophase} to be a self-contained equation, it remains to determine $f^-$ from $\eta$. From \eqref{eq:f-:0} we see that $f^-$ is a fixed point of  
\bq
\begin{aligned}
\mathcal{K}(\eta)g=\frac{1}{\mu^++\mu^-}(\mu^-|D_x|^{-1}R^+(\eta)-\mu^+|D_x|^{-1}R^-(\eta))g-\frac{\lb \rho\rb \mu^-}{\mu^++\mu^-}|D_x|^{-1}G^+(\eta)\eta.
\end{aligned}
\eq
\begin{prop}\label{prop:f-}
Let $r\ge 0$ and let $\eta\in  \dot B^1_{\infty, 1}\cap\dot B^{r+1}_{\infty, 1}$. If $\eta$ satisfies \eqref{smallcd} then $\mathcal{K}(\eta)$ has a unique fixed point $f^-$ in $\dot B^1_{\infty, 1}\cap \dot B^{1+r}_{\infty, 1}$ and
\bq\label{est:f-}
 \| f^-\|_{\dot B^{r}_{\infty, 1}}\le C(d, r)\lb\rho\rb \| |D_x|\eta\|_{\dot B^{r}_{\infty, 1}}.
\eq
\end{prop}
\begin{proof}
We apply Lemma \ref{lemm:fp} with $E_1=\dot B^1_{\infty, 1}$, $E_2=\dot B^{1+r}_{\infty, 1}$ and
\[
u_0=-\frac{\lb \rho\rb\mu^-}{\mu^++\mu^-}|D_x|^{-1}G^+(\eta)\eta\in E_1\cap E_2
\]
since, under condition \eqref{smallcd}, \eqref{est:G} gives
\[
 \| u_0\|_{\dot B^{1+r}_{\infty, 1}}\le  C\lb \rho\rb\| |D_x|\eta\|_{\dot B^r_{\infty, 1}},\quad C=C(d, r).
\]
 It then suffices to prove that the mapping 
\[
g\mapsto \cK_1(g):= \frac{1}{\mu^++\mu^-}(\mu^-|D_x|^{-1}R^+(\eta)-\mu^+|D_x|^{-1}R^-(\eta))g
\]
satisfies \eqref{boundE1} and \eqref{boundE2}. Indeed, under condition \eqref{smallcd} with $c_*$ smaller if necessary, we can apply \eqref{w0}  to have
\begin{align*}
\| \cK_1(\eta)g\|_{\dot B^{1+r}_{\infty, 1}}&\le \| |D_x|^{-1}R^-(\eta)\|_{\dot B^{1+r}_{\infty, 1}}+\| |D_x|^{-1}R^+(\eta)\|_{\dot B^{1+r}_{\infty, 1}}\\
&\le C\| |D_x|\eta\|_{\dot B^0_{\infty, 1}}\| g\|_{\dot B^{1+r}_{\infty, 1}}+C\| |D_x|\eta\|_{\dot B^r_{\infty, 1}}\| g\|_{\dot B^1_{\infty, 1}},\quad C=C(d, r).
\end{align*}
Thus there exists a unique fixed point $g\in \dot B^1_{\infty, 1}\cap \dot B^{1+r}_{\infty, 1}$ for $\cK$; moreover,
\bq\label{est:CK}
\| g\|_{\dot B^{1+r}_{\infty, 1}}\le C\lb \rho\rb\| |D_x|\eta\|_{\dot B^r_{\infty, 1}}+C\| |D_x|\eta\|_{\dot B^0_{\infty, 1}}\| g\|_{\dot B^{1+r}_{\infty, 1}}+C\| |D_x|\eta\|_{\dot B^r_{\infty, 1}}\| g\|_{\dot B^0_{\infty, 1}}.
\eq
 When $r=0$, combining \eqref{est:CK} and \eqref{smallcd} yields
\[
\| g\|_{\dot B^0_{\infty, 1}}\le C(d)\lb \rho\rb\| |D_x|\eta\|_{\dot B^0_{\infty, 1}}.
\]
Plugging this into \eqref{est:CK} and using \eqref{smallcd} again we find that for $r\ge 0$,
\[
\| g\|_{\dot B^{r}_{\infty, 1}}\le C(d, r)\lb \rho\rb\| |D_x|\eta\|_{\dot B^r_{\infty, 1}}
\]
which completes the proof.
\end{proof}
Assume that $\eta_1$ and $\eta_2$ satisfy \eqref{smallcd}, and $f^-_j$  is the fixed point of $\mathcal{K}(\eta_j)$. Set $\eta_\delta=\eta_1-\eta_2$ and $f_\delta=f_1-f_2$. It follows that
\begin{align*}
\| |D_x|f_\delta\|_{\B^r}&\le \| R^+(\eta_1)f^-_1-R^+(\eta_2)f^-_2\|_{\B^r}+\| R^-(\eta_1)f^-_1-R^-(\eta_2)f^-_2\|_{\B^r}\\
&\quad+\ka \| G^+(\eta_1)\eta_1-R^+(\eta_2)\eta_2\|_{\B^r}.
\end{align*}
Using the boundedness \eqref{w0} for $R^\pm(\eta_j)$ and the contraction \eqref{est:contraDN} for $R^\pm(\eta_1)-R^\pm(\eta_2)$, we easily deduce
\bq\label{contraction:f-}
\| |D_x|f_\delta\|_{\B^r}\le C\lb \rho \rb\big(\mathcal{A}_r\| |D_x|\eta_\delta\|_{\B^0}+\| |D_x|\eta_\delta\|_{\B^r}\big),
\eq
where $\mathcal{A}_r$ is given by \eqref{def:As}.

With $f^-$ given in Proposition \ref{prop:f-} and satisfying the estimates \eqref{est:f-} and \eqref{contraction:f-}, the nonlinearities $R^\pm(\eta)f^-$ and  $R^+(\eta)\eta$ in \eqref{PDE:twophase} obey the same bounds as for $R^-(\eta)\eta$ in the one-phase equation \eqref{PDE:onephase}. Therefore, the proof of Lemma \ref{proof:onephase} applies, leading to the global well-posedness for the the two-phase problem with small data in $\B^1$. 

\vspace{.1in}
\noindent{\bf{Acknowledgment.}} H. Q. Nguyen was partially supported by NSF grant DMS-190777. The author would like to thank Benoit Pausader for generously sharing many ideas. 

 
\end{document}